\documentclass[a4paper]{amsart}
\usepackage{graphicx}
\usepackage{amssymb}
\usepackage{amsmath}
\usepackage{amsthm}
\usepackage{amscd}
\usepackage{color}
\usepackage{colordvi}
\usepackage[all,2cell]{xy}
\usepackage{picture}
\usepackage{multirow}
\usepackage{makecell}

\usepackage[pagebackref,colorlinks]{hyperref}

\usepackage{amsfonts,enumerate,verbatim,mathtools,tikz,bm,tikz-cd,hyperref,comment}
\hypersetup{
	colorlinks=true,
	linkcolor=blue,
	filecolor=magenta,
	urlcolor=cyan,
}

\UseAllTwocells \SilentMatrices
\newtheorem{thm}{Theorem}[section]

\newtheorem{lem}[thm]{Lemma}
\newtheorem{exm}[thm]{Example}

\newtheorem{prop}[thm]{Proposition}
\theoremstyle{definition}
\newtheorem{defn}[thm]{Definition}
\theoremstyle{remark}
\newtheorem{rem}[thm]{\bf Remark}
\numberwithin{equation}{section}

\begin{document}
\title[Preprojective, skew group algebras and Morita equiv]{Preprojective algebras, skew group algebras and Morita equivalences}
\author[Xiao-Wu Chen, Ren Wang] {Xiao-Wu Chen, Ren Wang}

\thanks{}
\date{\today}

\thanks{E-mail: xwchen$\symbol{64}$mail.ustc.edu.cn; renw$\symbol{64}$mail.ustc.edu.cn}

\maketitle

\dedicatory{}%
\commby{}%

\begin{abstract}
Let $\mathbb{K}$ be  a field of characteristic $p$ and $G$ be a cyclic $p$-group which acts on a finite acyclic quiver $Q$. The folding process associates a Cartan triple to the action. We establish a Morita equivalence between the skew group algebra of the preprojective algebra of $Q$  and the generalized preprojective algebra associated to  the Cartan triple in the sense of Geiss, Leclerc and Schr\"{o}er. The Morita equivalence induces an isomorphism between certain ideal monoids of these preprojective algebras, which is compatible with the embedding of Weyl groups appearing in the folding process. 
 \end{abstract}

\section{Introduction}

Weyl groups appear in the representation theory of quivers with different incarnations \cite{BGP, Kac80, IT, ORT}. There is a remarkable bijection \cite{BIRS, Mizuno} between the Weyl group of a  finite acyclic quiver and a certain ideal monoid of  the preprojective algebra. Similarly, for a symmetrizable generalized Cartan matrix, there is a bijection \cite{FG} of the Weyl group and a certain ideal monoid of the generalized preprojective algebra in the sense of \cite{GLS17}. The motivation is to compare these two bijections via the folding process.

Let $\mathbb{K}$ be a field. Let $Q$ be a finite acyclic quiver. Denote by $W(Q)$ its Weyl group and by $\Pi(Q)$ its preprojective algebra \cite{GP, Rin}.  Each vertex $i\in Q_0$ gives rise to an idempotent $e_i$ in $\Pi(Q)$. Denote by $I_i$ the two-sided ideal of $\Pi(Q)$ generated by $1-e_i$, and by $\langle I_i\; |\; i\in Q_0\rangle$ the monoid generated by these ideals. The bijection establised in \cite{BIRS, Mizuno}
$$\Theta_Q\colon W(Q)\longrightarrow \langle I_i\; |\; i\in Q_0\rangle$$
sends  the simple reflection $s_i$  to $I_i$ for each $i\in Q_0$.

Let $(C, D, \Omega)$ be a Cartan triple, that is,  $C$ is a symmetrizable generalized Cartan matrix \cite{Kac}, $D$ is the symmetrizer and $\Omega$ is an acyclic orientation of $C$. Assume that the rows and columns of $C$ are indexed by a set $\Lambda$. Denote by $W(C)$ the Weyl group and by $\Pi(C, D, \Omega)$ the generalized preprojective algebra in the sense of \cite{GLS17}. For each $\mathbf{j}\in \Lambda$, we denote by $e_\mathbf{j}$ the corresponding idempotent in $\Pi(C, D, \Omega)$. Denote by $L_{\mathbf{j}}$ the two-sided ideal of $\Pi(C, D, \Omega)$ generated by $1-e_\mathbf{j}$. The bijection established in  \cite{FG}
$$\Theta_C\colon W(C)\longrightarrow \langle L_\mathbf{j}\; |\; \mathbf{j}\in \Lambda \rangle$$ 
sends the simple reflection $r_\mathbf{j}$ to $L_\mathbf{j}$ for each $\mathbf{j}\in \Lambda$.

The folding process is classic in Lie theory \cite{St} and plays a role in the representation theory of quivers \cite{Tani, Hu04}. Let $G$ be a finite group which acts on $Q$ by quiver automorphisms. One associates a Cartan triple $(C, D, \Omega)$ to the $G$-action. The rows and colomns of both $C$ and $D$ are indexed by the orbit set $Q_0/G$ of vertices in $Q$, and the entries of the digaonal matrix $D$ are the cardinalities  of certain stabilizers. By \cite{St, Hee}, there is a well-known isomorphism 
$$\psi\colon W(C)\longrightarrow  W(Q)^G,$$ 
which sends simple reflections $r_\mathbf{j}$ to $\prod_{i\in \mathbf{j}}s_i$. Here, $W(Q)^G$ denotes the subgroup formed by $G$-invariant elements in $W(Q)$. 

In view of the bijections $\Theta_Q, \Theta_C$ and the isomorphism $\psi$, it is natural to expect that the monoids $\langle I_i\; |\; i\in Q_0\rangle^G$  and $\langle L_\mathbf{j}\; |\; \mathbf{j}\in Q_0/G\rangle$ are isomorphic. We confirm this expectation; see Proposition~\ref{prop:Iyama-FG}. 

\vskip 5pt

\noindent {\bf Proposition~A.}  \emph{ Let $G$ be a finite group which acts on a finite acyclic quiver $Q$. Consider the associated Cartan triple $(C, D, \Omega)$. Then there is a unique isomorphism $\Psi$ between monoids making the following diagram commute. 
\[\xymatrix{
W(C)\ar[rr]^\psi \ar[d]_-{\Theta_C} && W(Q)^G \ar[d]^-{\Theta_Q^G}\\
\langle L_\mathbf{j}\; |\; \mathbf{j}\in Q_0/G\rangle  \ar[rr]^-{\Psi} && \langle  I_i\; |\; i\in Q_0\rangle^G
}\]
Here, $\Theta_Q^G$ denotes the restriction of $\Theta_Q$ on $W(Q)^G$. 
}

\vskip 5pt
We mention a similar commutative diagram in Remark~\ref{rem:monoid}, where we replace $\langle L_\mathbf{j}\; |\; \mathbf{j}\in Q_0/G\rangle$ by a certain ideal monoid $\langle I_\mathbf{j}\#G\; |\; \mathbf{j}\in Q_0/G\rangle$ of the skew group algebra $\Pi(Q)\#G$.

The isomorphism $\Psi$ suggests that the two preprojective algebras $\Pi(Q)$ and $\Pi(C, D, \Omega)$ might be closely related. The aim of this work is to relate these two preprojective algebras in a specific situation; see Theorem~\ref{thm:main}. 

\vskip 5pt
\noindent {\bf Theorem~B.}  \emph{Assume that ${\rm char}(\mathbb{K})=p>0$ and that $G$ is a cyclic $p$-group. Assume that the $G$-action on $Q$ satisfies $G_\alpha=G_{s(\alpha)}\cap G_{t(\alpha)}$ for any arrow $\alpha$ in $Q$. Then there is a Morita equivalence
$$F\colon \Pi(Q)\#G \mbox{-}{\rm Mod}\longrightarrow \Pi(C, D, \Omega)\mbox{-}{\rm Mod}$$
such that
\begin{align}\label{equ:categorify}
\Psi^{-1}(I)=\Phi_F(I\#G)
\end{align}
for each $I\in \langle  I_i\; |\; i\in Q_0\rangle^G$. }

\vskip 5pt

Here, $\Phi_F$ denotes the isomorphism between the ideal monoids induced by the Morita equivalence $F$; see Proposition~\ref{prop:MI-iso}. For each arrow $\alpha$ with the starting vertex $s(\alpha)$ and terminating vertex $t(\alpha)$, we denote by $G_\alpha$, $G_{s(\alpha)}$ and $G_{t(\alpha)}$ their  stabilizers.  

The identity (\ref{equ:categorify}) indicates that, in a certain sense, the isomorphism $\Psi^{-1}$ is \emph{categorified}  by the Morita equivalence $F$ and the induction functor $-\#G$.

The proof of Proposition~A relies on the fact that these ideal monoids are isomorphic to the corresponding Weyl monoids \cite{Ts}. Moreover, we establish an analogue of the isomorphism $\psi$ for the Weyl monoids in Proposition~\ref{prop:monoid}. The proof of Theorem~B relies on the Morita equivalence in \cite{CW24} between the skew goup algebra of $\mathbb{K}Q$ and the algebra $H(C, D, \Omega)$ in \cite{GLS17}. We also use general results on $2$-preprojective algebras of   arbitrary algebras, which are implicit in \cite{Le}. When $Q$ is of type $A$ and $G$ is of order $2$, such a  Morita equivalence $F$ is also established in \cite{KKKMM24}.

The paper is structured as follows. We recall the definition  of $n$-preprojective algebras of an arbitrary algebra in Section~2. We study the skew group algebras, compatible bimodules and their induced bimodules in Section~3. We prove that the skew group algebra of a preprojective algebra is isomorphic to the preprojective algebra of the skew group algebra in Theorem~\ref{thm:pi}. In Section~4, we prove that any Morita equivalence between two algebras extends naturally to a Morita equivalence between their preprojective algebras; see Proposition~\ref{prop:Pi}. We study Weyl groups and Weyl monoids \cite{Ts} associated to quivers and Cartan matrices in Section~5. 

We recall the generalized preprojective algebra  and prove Proposition~A in Section~6. In final section, we prove Theorem~B, which is illustrated with an explicit  example.

Throughout this paper, we work over a field $\mathbb{K}$. Unadorned Hom and tensor functors are all over $\mathbb{K}$.  By default, a module means a unital left module. For an algebra $A$, we denote by $A\mbox{-Mod}$ the category of all $A$-modules.

\section{Tensor algebras and  preprojective algebras}

In this section, we fix the notation  and recall some well-known facts on tensor algebras and preprojective algebras.

Let $A$ be an algebra, and $V$ an $A$-$A$-bimodule on which $\mathbb{K}$ acts centrally. The associated \emph{tensor algebra} $T_A(V)$ is given by 
$$T_A(V)=A\oplus V\oplus V^{\otimes_A 2}\oplus \cdots \oplus V^{\otimes_A n}\oplus \cdots$$
where $V^{\otimes_A n}$ is the $n$-fold tensor product of $V$. 

Let $\mathcal{C}$ be any category with an endofunctor $E\colon \mathcal{C}\rightarrow \mathcal{C}$. By a \emph{representation} of $E$, we mean a pair $(X, \alpha)$ consisting of an object $X$ in $\mathcal{C}$ and a morphism $\alpha\colon E(X)\rightarrow X$. A morphism $f\colon (X, \alpha)\rightarrow (Y, \beta)$ between representations is given by a morphism $f\colon X\rightarrow Y$ in $\mathcal{C}$ satisfying
$$f\circ \alpha=\beta\circ E(f).$$
These data form the category of representations of $E$, denoted by ${\rm rep}(E)$; see \cite[p.469]{Rin} and  \cite[Subsection~2.1]{CL}. 

The following fact is standard. 

\begin{lem}\label{lem:rep}
Let $\mathcal{C}'$ be another category with an endofunctor $E'$. Assume that there is an equivalence $F\colon \mathcal{C}\rightarrow \mathcal{C}'$ of categories such that $FE$ is isomorphic to $E'F$. Then we have an induced equivalence $\widetilde{F}\colon {\rm rep}(E)\rightarrow {\rm rep}(E')$. 
\end{lem}

\begin{proof}
Assume that we are given a natural isomorphism $\eta\colon E'F\rightarrow FE$. The induced functor $\widetilde{F}$ sends a representation $(X, \alpha)$ of $E$ to $(F(X),  F(\alpha)\circ \eta_X)$ of $E'$. Similarly, one constructs a quasi-inverse of $\widetilde{F}$.
\end{proof}

Let $A$ be an algebra, and $V$ an $A$-$A$-bimodule. Consider the endofunctor 
$$V\otimes_A-\colon A\mbox{-Mod}\longrightarrow A\mbox{-Mod}.$$
There is an isomorphism of categories
\begin{align}\label{iso:TAV}
T_A(V)\mbox{-Mod}\longrightarrow {\rm rep}(V\otimes_A-),
\end{align}
which sends a $T_A(V)$-module $X$ to the representation $(X, \alpha)$, where $X$ is the underlying $A$-module and $\alpha(v\otimes_A x)=vx$ for $v\in V$ and $x\in X$; compare \cite[Lemma~2]{Rin}.  

Denote by $A^e=A\otimes A^{\rm op}$ the \emph{enveloping algebra} of $A$. We identify $A$-$A$-bimodules with left $A^e$-modules, which are also identified with right $A^e$-modules. 

Let $e$ and $f$ be two idempotents of $A$. Then $Ae\otimes fA$ is naturally an $A$-$A$-bimodule, which is cyclic and projective.  We have a canonical isomorphism of $A$-$A$-bimodules
\begin{align}\label{iso:bi-dual}
    {\rm Hom}_{A^e}(Ae\otimes fA, A^e)\longrightarrow Af\otimes eA,\; \theta\mapsto {\rm swap}(\theta(e\otimes f)).
\end{align}
Here, ${\rm swap}(a\otimes b)=b\otimes a$, and the $A$-$A$-bimodule structure of  ${\rm Hom}_{A^e}(Ae\otimes fA, A^e)$ is induced by the inner $A$-$A$-bimodule structure on $A^e$. Similarly, for each $n\geq 1$, the $n$-th extension space ${\rm Ext}^n_{A^e}(A, A^e)$ is naturally a right $A^e$-module, and thus an $A$-$A$-bimodule. 

The following fact is well known.

\begin{lem}\label{lem:fd}
Assume that the algebra $A$ is finite dimensional. Then we have an isomorphism of $A$-$A$-bimodules
$${\rm Ext}^n_{A^e}(A, A^e)\simeq {\rm Ext}_A^n(DA, A) $$
for each $n\geq 1$. 
\end{lem}

\begin{proof}
The canonical map $A^e\rightarrow {\rm Hom}(DA, A)$, sending $a\otimes b$ to $(\theta\mapsto \theta(b)a)$, is an isomorphism of left $A^e$-modules. Therefore, we identify ${\rm Ext}^n_{A^e}(A, A^e)$ with ${\rm Ext}^n_{A^e}(A, {\rm Hom}(DA, A))$, which is canonically isomorphic to ${\rm Ext}^n_{A}(DA, A)$. 
\end{proof}

In view of \cite[Theorem~A]{Rin} and Lemma~\ref{lem:fd}, the following definition is natural; compare \cite{GLS17}. We mention that its derived analogue is due to \cite[Subsection~4.1]{Ke11}.

\begin{defn}
Let $A$ be an algebra and $n\geq 2$. Then \emph{$n$-preprojective algebra} of $A$ is defined to be the tensor algebra $\Pi_n(A)=T_A({\rm Ext}^{n-1}_{A^e}(A, A^e))$.
\end{defn} 

 Let $Q=(Q_0, Q_1;s,t)$ be a finite quiver, where $Q_0$ is the finite set of vertices, $Q_1$ is the finite set of arrows and the two maps $s,t\colon Q_1\rightarrow Q_0$ assign to each arrow $\alpha$ its starting vertex $s(\alpha)$ and its terminating vertex $t(\alpha)$. A path $p=\alpha_n\cdots \alpha_2\alpha_1$ of length $n$ consists of $n$ consecutive arrows $\alpha_i$'s, that is, $t(\alpha_i)=s(\alpha_{i+1})$ for $1\leq i\leq n-1$. Here, we write the concatenation from right to left. We observe that a path of length $1$ is just an arrow. For each vertex $i$, we associate a trivial path $e_i$ of length $0$. Denote by $\mathbb{K}Q$ the \emph{path algebra}, which has a basis given by all paths in $Q$ and whose multiplication is given by the concatenation of paths. 

Denote by $\overline{Q}$ the \emph{double quiver} of $Q$, which is obtained from $Q$ by adding for each arrow $\alpha\in Q_1$ an inverse arrow $\alpha^*$. Following \cite{Rin}, the \emph{preprojective algebra} of $Q$ is defined by
$$\Pi(Q)=\mathbb{K}\overline{Q}/{(\sum_{\alpha\in Q_1} \alpha \alpha^*-\alpha^*\alpha)}.$$
We mention the implicit appearance of the preprojective algebra in \cite[Section~12]{Lu1991}.

The following result is well known; compare \cite[Theorem~A]{Rin}. 

\begin{lem}\label{lem:quiver}
Let $Q$ be any  finite quiver. Then there is an isomorphism $\Pi_2(\mathbb{K}Q)\simeq \Pi(Q)$ of algebras.
\end{lem}

\begin{proof}
     Write $A=\mathbb{K}Q$.    We have a canonical bimodule projective resolution of $A$.
  \begin{align}\label{equ:bimod-res}
  0\longrightarrow \bigoplus_{\alpha\in Q_1} A e_{t(\alpha)}\otimes \mathbb{K}\alpha \otimes e_{s(\alpha)}A\stackrel{\partial}\longrightarrow \bigoplus_{i\in Q_0} Ae_i\otimes e_iA \longrightarrow A\longrightarrow 0 \end{align}
   Here, $\partial(e_{t(\alpha)}\otimes \alpha \otimes e_{s(\alpha)})=\alpha\otimes e_{s(\alpha)}-e_{t(\alpha)}\otimes \alpha$ and the unnamed arrow on the right is given by the multiplication in $A$. Applying ${\rm Hom}_{A^e}(-, A^e)$ to this sequence and using the isomorphism (\ref{iso:bi-dual}), we infer that ${\rm Ext}^1_{A^e}(A, A^e)$ is isomorphic to the cokernel of the following morphism
  \begin{align}\label{equ:quiver}
      \bigoplus_{i\in Q_0} Ae_i\otimes e_iA   \stackrel{\partial'} \longrightarrow \bigoplus_{\alpha\in Q_1} A e_{s(\alpha)}\otimes \mathbb{K}\alpha^* \otimes e_{t(\alpha)}A,
  \end{align}
   which is given by 
   $$\partial'(e_i\otimes e_i)=\sum_{\{\alpha\in Q_1\;|\; s(\alpha)=i \}}e_{i}\otimes \alpha^*\otimes \alpha-\sum_{\{\beta\in Q_1\; |\; t(\beta)=i\}} \beta\otimes \beta^*\otimes e_{i}.$$
   Here for each $\alpha\in Q_1$, we identify $Ae_{t(\alpha)}\otimes \mathbb{K}\alpha\otimes e_{s(\alpha)}A$ with $Ae_{t(\alpha)}\otimes e_{s(\alpha)}A$, and $Ae_{s(\alpha)}\otimes \mathbb{K}\alpha^*\otimes e_{t(\alpha)}A$ with $Ae_{s(\alpha)} \otimes e_{t(\alpha)}A$.   
   
   We observe that the tensor algebra $$T_A(\bigoplus_{\alpha\in Q_1} A e_{s(\alpha)}\otimes \mathbb{K}\alpha^* \otimes e_{t(\alpha)}A)$$ is naturally isomorphic to the path algebra $\mathbb{K}\overline{Q}$ of the double quiver $\overline{Q}$. Then the required isomorphism follows immediately. 
\end{proof}

\section{Skew group algebras and induced bimodules}

In this section, we will recall basic facts on skew group algebras, and study their induced bimodules. We prove that the $n$-preprojective algebra of a skew group algebra is isomorphic to the skew group algebra of the $n$-preprojective algebra; see Theorem~\ref{thm:pi}.

We emphasize that many results in this section are implicitly due to \cite{Le}. We provide full proofs for completeness, since the setting there  is very different.  We mention related work on skew group algebras of  quiver algebras \cite{Dem, GP, Thi}.

\subsection{Compatible bimodules and induced bimodules} 

Let $A$ be an algebra. Denote by $\mathbf{I}(A)$ the \emph{ideal monoid} of $A$, which consists of two-sided ideals of $A$ and whose multiplication is given by the multiplication between ideals. Let $G$ be a finite group, which is written multiplicatively and whose identity is denoted by $1_G$. 

We fix a $G$-action $\rho\colon G\rightarrow {\rm Aut}(A)$ on $A$ by algebra automorphisms.  Write $\rho(g)(a)=g(a)$ for any $g\in G$ and $a\in A$. The \emph{skew group algebra} is given by $A\# G=A\otimes \mathbb{K}G$, whose typical element is denoted by $a\# g$ and whose multiplication is defined by
$$(a\# g)(b\# h)=ag(b)\# gh.$$
The following identity
\begin{align}\label{equ:skew}
(1_A\#g)(a\#1_G)(1_A\#g^{-1})=g(a)\# 1_G
\end{align}
will be used. We view $A$ as a subalgebra of $A\#G$ by identifying $a$ with $a\#1_G$. Consider the projection ${\rm pr}\colon A\#G\rightarrow A$, which sends $a\#g$ to $\delta_{g, 1_G}a$.

\begin{lem}\label{lem:Hom-skew}
    Let $M$ be a left $A\#G$-module. The the following two statements hold.
\begin{enumerate}
\item The space ${\rm Hom}_A(M, A)$ becomes a right $A\#G$-module in the following manner: for any $\theta\in {\rm Hom}_A(M, A)$, the element $\theta(a\#g)\in {\rm Hom}_A(M, A)$ sends each $m\in M$ to $g^{-1}(\theta((1_A\#g)m)a)$.  
    \item The projection above induces an isomorphism of right $A\#G$-modules
    $${\rm Hom}_{A\#G}(M, A\#G)\longrightarrow {\rm Hom}_A(M, A).$$
\end{enumerate}
\end{lem}

\begin{proof}
    The proof of (1) is routine. For (2), we refer to \cite[Subsection~1.1 (B)]{RR}. The inverse map sends $\theta\in {\rm Hom}_A(M, A)$ to $\theta'\colon M\rightarrow A\#G$, which is given by $(m\mapsto \sum_{g\in G} g\theta((1\#g^{-1})m)\#g)$; compare \cite[Lemma~1.2]{RR}.
\end{proof}

A two-sided ideal $I$ of $A$ is called \emph{$G$-invariant} if $g(I)=I$ for all $g\in G$. Denote by $\mathbf{I}(A)^G$ the sub monoid of $\mathbf{I}(A)$ formed by $G$-invariant ideals.

 We observe that the decomposition $A\#G=\oplus_{g\in G}(A\#g)$ makes $A\#G$ a  $G$-graded algebra. A two-sided ideal $J$ of $A\#G$ is \emph{$G$-graded} if $J=\oplus_{g\in G}(J\cap (A\#g))$. Denote by $\mathbf{I}(A\#G)_G$ the sub monoid of $\mathbf{I}(A\#G)$ formed by $G$-graded ideals.

The following results are elementary.

\begin{prop}\label{prop:inv-ideal}
 Let $I$ be a two-sided ideal of $A$ and $J$ a two-sided ideal of $A\#G$. Then the following results hold.
 \begin{enumerate}
  \item[(1)] The subspace $\{a\in A\; |\; a\#1_G\in J\}$ is a $G$-invariant two-sided ideal of $A$.
 \item[(2)] The subspace $I\#G=\{\sum_{g\in G}a_g\#g\; |\; a_g\in I\}$ is a two-sided ideal of $A\#G$ if and only if the ideal $I$ is $G$-invariant.
 \item[(3)] There is an isomorphism $\mathbf{I}(A)^G\rightarrow \mathbf{I}(A\#G)_G$ of monoids, sending any $G$-invariant ideal $I'$ to $I'\#G$.
 \end{enumerate}
\end{prop}

\begin{proof}
We use (\ref{equ:skew}) to verify (1). The ``only if" part of (2) follows from (1), and the ``if" part is straightforward.  

The isomorphism in (3) follows immediately  from (1) and (2), whose inverse map sends any $G$-graded ideal $J'$ of $A\#G$ to $\{a\in A\; |\; a\#1_G\in J'\}$, which belongs to $\mathbf{I}(A)^G$ by (1).
\end{proof}

The group $G$ acts on the enveloping algebra $A^e$ by algebra automorphisms, that is, $g(a\otimes b)=g(a)\otimes g(b)$. Denote by $\Delta=A^e\#G$ the corresponding skew group algebra. We observe an algebra embedding 
\begin{align}\label{diag-sub}
    \Delta\hookrightarrow (A\#G)^e, (a\otimes b)\#g\mapsto (a\#g)\otimes (g^{-1}(b)\#g^{-1}).
\end{align}
In view of this embedding, we might call $\Delta$   the \emph{diagonal subalgebra} of $(A\#G)^e$; compare \cite[Subsection~3.1]{Le}.   We observe that $(A\otimes G)^e$ is free both  as a left $\Delta$-module and a right $\Delta$-module. 
 We have an isomorphism of algebra
\begin{align}\label{iso:diag}
    \Delta\longrightarrow \Delta^{\rm op}, \; (a\otimes b)\#g\mapsto (g^{-1}(b)\otimes g^{-1}(a))\#g^{-1}.
\end{align}

The following notion is implicitly due to \cite[Subsection~3.1]{Le}.

\begin{defn}
By a \emph{$G$-compatible $A$-$A$-bimodule}, we mean an $A$-$A$-bimodule $M$ with a $\mathbb{K}$-linear $G$-action satisfying that
$$g(amb)=g(a)g(m)g(b)$$
for any $a, b\in A$ and $m\in M$. Here, we denote by $g(amb)$ and $g(m)$ the $g$-actions on the elements $amb$ and $m$, respectively.
\end{defn}

 \begin{rem}
     We observe that a $G$-compatible $A$-$A$-bimodule structure is equivalent to a left $\Delta$-module structure, and is also equivalent to a right $\Delta$-module structure; compare (\ref{iso:diag}). More precisely, let $M$ be a $G$-compatible $A$-$A$-bimodule. Then $M$ is naturally a left $\Delta$-module with the following action
$$((a\otimes b)\#g) m= ag(m)b.$$
Similarly, $M$ is naturally a right $\Delta$-module given by
$$m((a\otimes b)\#g) = g^{-1}(bma).$$
 \end{rem}

\begin{defn}
Let $M$ be a $G$-compatible $A$-$A$-bimodule. The \emph{induced bimodule} $M\#G=M\otimes \mathbb{K}G$ is a bimodule over $A\#G$ defined as follows
$$(a\#g)(m\#h)(a'\#g')=ag(m)gh(a')\#ghg'$$
for any $a, a'\in A$, $m\in M$ and $g, h, g'\in G$.
\end{defn}

The following remark justifies the terminology above.

\begin{rem}\label{rem:induced}
Let $M$ be a $G$-compatible $A$-$A$-bimodule. Then there is an isomorphism
$$(A\#G)\otimes_A M\longrightarrow M\#G, \quad (a\#g)\otimes_A m\mapsto ag(m)\#g$$
of left $A\#G$-modules. Therefore, the left $A\#G$-module structure on $M\#G$ is induced from the left $A$-module structure on $M$. A similar remark works on the right side. 
\end{rem}

\begin{exm}
{\rm Let $I$ be a $G$-invariant two-sided ideal of $A$. Then $G$ acts on $I$. Therefore, as an $A$-$A$-bimodule, $I$ is $G$-compatible. We observe that the corresponding induced bimodule structure on $I\#G$ is the same as the one inherited from the two-sided ideal $I\#G$ of $A\#G$.}
\end{exm}

The following results are essentially due to \cite[Lemma~3.1.1]{Le}.

\begin{lem}\label{lem:induced}
    Let $M$ be a $G$-compatible $A$-$A$-bimodule. Then the following statements hold.
    \begin{enumerate}
        \item There is an isomorphism $(A\# G)^e\otimes_\Delta M\simeq M\#G$
        of left $(A\#G)^e$-modules.
        \item There is an isomorphism $M\otimes_\Delta (A\# G)^e \simeq M\#G$
        of right $(A\#G)^e$-modules.
    \end{enumerate}
\end{lem}

In (1), we view $M$ as a left $\Delta$-module. Then using the algebra embedding (\ref{diag-sub}), we have the induced left module $(A\# G)^e\otimes_\Delta M$. A similar remark holds for (2). 

\begin{proof}
 The isomorphism in (1) sends $((a\#g)\otimes (b\#h))\otimes_\Delta m$ to $ag(mb)\#gh$. The isomorphism in (2) sends $m\otimes_\Delta ((a\#g)\otimes (b\#h))$ to $bh(ma)\#hg$. We omit the details. 
\end{proof}

Let $M$ and $N$ be two $G$-compatible $A$-$A$-bimodules. Then the $A$-$A$-bimodule $M\otimes_A N$ is $G$-compatible with the diagonal $G$-action. The proof of the following result is routine. 

\begin{lem}\label{lem:ind-tensor}
Let $M$ and $N$ be two $G$-compatible $A$-$A$-bimodules. Then there is an isomorphism of $(A\#G)$-$(A\#G)$-bimodules
$$(M\otimes_A N)\#G \longrightarrow (M\#G)\otimes_{A\#G} (N\#G),$$
 which sends $(m\otimes_A n)\#g$ to $(m\#1_G)\otimes_{A\#G}(n\#g)$.\hfill $\square$
\end{lem}

Let $V$ be a $G$-compatible $A$-$A$-bimodule. Then $G$ acts naturally on the tensor algebra $T_A(V)$ by algebra automorphisms in the following manner: for $g\in G$ and $v_1\otimes_A v_2\otimes_A \cdots \otimes_A v_n\in V^{\otimes_A n}$, we define 
$$g(v_1\otimes_A v_2\otimes_A \cdots \otimes_A v_n)=g(v_1)\otimes_A g(v_2)\otimes_A \cdots \otimes_A g(v_n).$$
We form the skew group algebra $T_A(V)\#G$.

\begin{prop}\label{prop:tensor-skew}
Let $V$ be a $G$-compatible $A$-$A$-bimodule with $V\#G$ the induced bimodule over $A\#G$. Then we have an isomorphism $T_A(V)\#G\simeq T_{A\#G}(V\#G)$ of algebras.
\end{prop}

\begin{proof}
By applying Lemma~\ref{lem:ind-tensor} repeatedly, we infer  an isomorphism 
$$\phi_n\colon V^{\otimes_A n}\#G\longrightarrow (V\#G)^{\otimes_{A\#G} n}$$
for each $n\geq 1$, which sends $(v_1\otimes_A v_2\otimes_A \cdots \otimes_A v_n)\#g$ to 
$$(v_1\#1_G)\otimes_{A\#G} (v_2\#1_G)\otimes_{A\#G} \cdots \otimes_{A\#G}(v_n\#g). $$
It is direct to verify that these isomorphisms $\phi_n$ give rise to the required isomorphism of algebras. 
\end{proof}

\subsection{Skew group algebras of preprojective algebras}

Let $M$ be a $G$-compatible $A$-$A$-bimodule. Consider the dual $A$-$A$-bimodule ${\rm Hom}_{A^e}(M, A^e)$. It is naturally $G$-compatible by the contragredient  $G$-action: for each $g\in G$ and $\theta\colon M\rightarrow A^e$, we define $g(\theta)\colon M\rightarrow A^e$ such that $g(\theta)(m)=g(\theta(g^{-1}(m)))$. Then we have the induced bimodule ${\rm Hom}_{A^e}(M, A^e)\#G$ over $A\#G$. 

The following result is essentially due to \cite[Lemma~3.3.1]{Le}, where the diagonal subalgebra $\Delta$ plays a role.

\begin{lem}\label{lem:bimod-dual-skew}
   Let $M$ be a $G$-compatible $A$-$A$-bimodule. Then we have an isomorphism of $(A\#G)$-$(A\#G)$-bimodules
   $${\rm Hom}_{A^e}(M, A^e)\#G\simeq {\rm Hom}_{(A\#G)^e}(M\#G, (A\#G)^e).$$
\end{lem}

\begin{proof}
    Recall that $\Delta=A^e\#G$. By Lemma~\ref{lem:induced}(1), we identify the induced bimodule $M\#G$ with $(A\#G)^e\otimes_\Delta M$. By the Hom-tensor adjunction, we have the first isomorphism in the following identity.
    \begin{align*}
        {\rm Hom}_{(A\#G)^e}(M\#G, (A\#G)^e) &\simeq {\rm Hom}_{\Delta}(M, (A\#G)^e)\\
                                            & \simeq {\rm Hom}_{\Delta}(M, \Delta)\otimes_\Delta (A\#G)^e\\
                                            &\simeq {\rm Hom}_{A^e}(M, A^e)\otimes_\Delta (A\#G)^e\\
                                            &\simeq {\rm Hom}_{A^e}(M, A^e)\#G
    \end{align*}
    Here, the second isomorphism follows since $(A\#G)^e$ is a finitely generated free left $\Delta$-module, the third one follows by applying Lemma~\ref{lem:Hom-skew}(2) to the left $\Delta$-module $M$, and the last one follows from Lemma~\ref{lem:induced}(2).   
\end{proof}

 Let $M$ be a $G$-compatible $A$-$A$-bimodule. For each $n\geq 1$, we observe that the $A$-$A$-bimodule ${\rm Ext}^n_{A^e}(M, A^e)$ is naturally $G$-compatible. Indeed, we take a projecive resolution $P^\bullet$ of $M$ as a left $\Delta$-module. Therefore, each component $P^{-n}$ is a $G$-compatible $A$-$A$-bimodule, whose underlying $A$-$A$-bimodule is projective. Each component of the dual complex ${\rm Hom}_{A^e}(P^\bullet, A^e)$ is naturally $G$-compatible. Consequently, 
 $${\rm Ext}^n_{A^e}(M, A^e)=H^n({\rm Hom}_{A^e}(P^\bullet, A^e))$$
 is a naturally $G$-compatible. 

 \begin{lem}\label{lem:Ext-bi}
   Let $M$ be a $G$-compatible $A$-$A$-bimodule. Then for each $n\geq  1$, we have an isomorphism of $(A\# G)$-$(A\# G)$-bimodules   
   $${\rm Ext}^n_{A^e}(M, A^e)\#G \simeq {\rm Ext}^n_{(A\#G)^e}(M\#G, (A\#G)^e).$$
\end{lem}

\begin{proof}
By  ${\rm Ext}^n_{A^e}(M, A^e)=H^n({\rm Hom}_{A^e}(P^\bullet, A^e))$, we infer that 
$${\rm Ext}^n_{A^e}(M, A^e)\#G =H^n({\rm Hom}_{A^e}(P^\bullet, A^e)\#G).$$
Lemma~\ref{lem:bimod-dual-skew} implies that there is an isomorphism of complexes
$${\rm Hom}_{A^e}(P^\bullet, A^e)\#G \simeq {\rm Hom}_{(A\#G)^e}(P^\bullet\#G, (A\#G)^e).$$
We observe that $P^\bullet\#G$ is a projective resoluion of the induced bimodule $M\#G$. Then we infer an isomorphism
$$H^n({\rm Hom}_{A^e}(P^\bullet, A^e)\#G)\simeq {\rm Ext}^n_{(A\#G)^e}(M\#G, (A\#G)^e).$$
This completes the proof.
\end{proof}

Let $G$ be a finite group and $A$ an algebra with a $G$-action. For $n\geq 2$,  the $A$-$A$-bimodule ${\rm Ext}^{n-1}_{A^e}(A, A^e)$ is naturally $G$-compatible. Therefore, the $n$-preprojective algebra $\Pi_n(A)=T_A({\rm Ext}^{n-1}_{A^e}(A, A^e))$ has an induced $G$-action. We form the skew group algebra $\Pi_n(A)\#G$.  

The derived analogue of the following result is due to \cite[Theorem~3.5.4(1)]{Le}.

\begin{thm}\label{thm:pi}
Keep the assumptions above. Then we have an isomorphism $\Pi_n(A)\#G\simeq \Pi_n(A\#G)$ of algebras.    
\end{thm}

\begin{proof}
By Lemma~\ref{lem:Ext-bi}, we have an isomorphism of $(A\#G)$-$(A\#G)$-bimodules.
$${\rm Ext}_{(A\#G)^e}^{n-1}(A\#G, (A\#G)^e) \simeq {\rm Ext}^{n-1}_{A^e}(A, A^e)\#G$$
It follows that $\Pi_{n}(A\#G)$ is isomorphic to the following tensor algebra
$$T_{A\#G}({\rm Ext}^{n-1}_{A^e}(A, A^e)\#G).$$
By Proposition~\ref{prop:tensor-skew}, the algebra above is isomorphic to
$$T_A({\rm Ext}^{n-1}_{A^e}(A, A^e))\#G=\Pi_n(A)\#G.$$
This completes the proof. 
\end{proof}

Let $G$ be a finite group which acts on a finite quiver $Q$ by quiver automorphisms. Then $G$ acts on the path algebra $\mathbb{K}Q$ and the preprojective algebra $\Pi(Q)$. Here, we observe that $g(\alpha^*)=g(\alpha)^*$ for $g\in G$ and $\alpha\in Q_1$. We form the corresponding skew group algebras $\mathbb{K}Q\#G$ and $\Pi(Q)\#G$. 

\begin{prop}
	Keep the assumptions above. Then we have an isomorphism $\Pi(Q)\#G\simeq \Pi_2(\mathbb{K}Q\#G)$ of algebras.
\end{prop}

\begin{proof}
	The $G$-action on $\mathbb{K}Q$ induces a $G$-action on the $2$-preprojective algebra $\Pi_2(\mathbb{K}Q)$. By Lemma~\ref{lem:quiver}, we have an isomorphism $\Pi(Q)\simeq \Pi_2(\mathbb{K}Q)$. We claim that this isomorphism is compatible with the two $G$-actions.     
	
	For the claim, it suffices to prove that the following statement is true: the induced $G$-action on ${\rm Ext}^1_{(\mathbb{K}Q)^e}(\mathbb{K}Q, (\mathbb{K}Q)^e)$ coincides with the one given by $g(\alpha^*)=g(\alpha)^*$. Indeed, the projective resolution (\ref{equ:bimod-res}) is a complex of $G$-compatible $\mathbb{K}Q$-$\mathbb{K}Q$-bimodules.  Applying ${\rm Hom}_{(\mathbb{K}Q)^e}(-, (\mathbb{K}Q)^e)$ to (\ref{equ:bimod-res}), we obtain (\ref{equ:quiver}), whose induced $G$-action is given as follows:
	$g(ae_i\otimes e_ib)=g(a)e_{g(i)}\otimes e_{g(i)}g(b)$ and 
	$$g(ae_{s(\alpha)}\otimes \alpha^*\otimes e_{t(\alpha)}b)=g(a)e_{s(g(\alpha))}\otimes g(\alpha)^*\otimes e_{t(g(\alpha))}g(b)$$
	for each $i\in Q_0$ and $\alpha\in Q_1$. Since ${\rm Ext}^1_{(\mathbb{K}Q)^e}(\mathbb{K}Q, (\mathbb{K}Q)^e)$ is identified with the cokernel of $\partial'$, the statement above holds.  
	
	The claim above implies an isomorphism $\Pi(Q)\#G\simeq \Pi_2(\mathbb{K}Q)\#G$. Then the required isomorphism follows from Theorem~\ref{thm:pi} immediately. 
\end{proof}

\section{Morita equivalences and bimodules}

In  this section, we recall known facts on Morita equivalences between two algebras. We prove that any Morita equivalence between two algebras extends to a Morita equivalence between their preprojective algebras; see Proposition~\ref{prop:Pi}.

Let $A$ and $B$ be two algebras. We fix a $\mathbb{K}$-linear Morita equivalence 
$$F\colon A\mbox{-Mod}\longrightarrow B\mbox{-Mod}$$ between $A$ and $B$. There is an invertible $B$-$A$-bimodule $P$ such that $F\simeq P\otimes_A-$. Moreover, the bimodule $P$ fits into a set  $(P, Q; \phi,\psi)$ of \emph{equivalence data}, where $Q$ is an $A$-$B$-bimodule, $\phi\colon P\otimes_A Q\rightarrow B$ is an isomorphism of $B$-$B$-bimodules, and $\psi\colon Q\otimes_B P\rightarrow A$ is an isomorphism of $A$-$A$-bimodules. Moreover, these data satisfy the associativity condition, that is, 
\begin{align}\label{equ:asso}
\phi(x\otimes_A y)x'=x\psi(y\otimes_Bx') \mbox{ and } y'\phi(x\otimes_A y)=\psi(y'\otimes_B x)y,
\end{align}
for any $x, x'\in P$ and $y, y'\in Q$. We refer to \cite[Chapter II, $\S$ 3]{Bass} for details.

The equivalence data induce another Morita equivalence
$$F^e=P\otimes_A-\otimes_AQ\colon A^e\mbox{-Mod}\longrightarrow B^e\mbox{-Mod},$$
whose quasi-inverse might be chosen as $Q\otimes_B-\otimes_B P$.

\begin{rem}
Denote by $\mathcal{E}^c(A\mbox{-Mod})$ the category of continuous endofunctors on $A\mbox{-Mod}$, that is, endofunctors which preserve arbitrary coproducts. By Eilenberg-Watt's theorem,  we have an equivalence 
$$A^e\mbox{-Mod}\longrightarrow \mathcal{E}^c(A\mbox{-Mod}),$$ 
sending any $A$-$A$-bimodule $X$ to the endofunctor $X\otimes_A-$. Then we have the following square, which commutes up to a natural isomorphism.
\[\xymatrix{
A^e\mbox{-Mod}\ar[d] \ar[rr]^{F^e} && B^e\mbox{-Mod} \ar[d]\\
\mathcal{E}^c(A\mbox{-Mod})\ar[rr] && \mathcal{E}^c(B\mbox{-Mod})
}\]
Here, the functor at the bottom sends $H$ to $FHF^{-1}$, where $F^{-1}$ is a quasi-inverse of $F$. This justifies the notation $F^e$ to some extent. 
\end{rem}

For each left $B$-module $Y$, we denote by ${\rm add}\; Y$ the full subcategory of $B\mbox{-Mod}$ formed by direct summands of finite sums of $Y$.

\begin{lem}\label{lem:Fe}
Let $X$ be an $A$-$A$-bimodule. Then we have ${\rm add}\; F(X)={\rm add} \; F^e(X)$ in $B\mbox{-}{\rm Mod}$.
\end{lem}

\begin{proof}
Recall that $F(X)=P\otimes_A X$ and $F^e(X)=(P\otimes_A X)\otimes_A Q$. Then the required equality follows from  the fact that the underlying left $A$-module $Q$ is a finitely generated projective generator. 
\end{proof}

Let $I$ be a two-sided ideal of $A$. The following subspace of $B$
\begin{align*}
\Phi_F(I) &=\{\mbox{finite sums } \sum_i\phi(x_i\otimes_Ay_i)\; |\; x_i\in PI, y_i\in Q\}\\
 &=\{\mbox{finite sums } \sum_i\phi(x_i\otimes_Ay_i)\; |\; x_i\in P, y_i\in IQ\}
\end{align*}
is a two-sided ideal of $B$.

The following result is implicit in \cite[Chapter II, Theorem~3.5(6)]{Bass}.

\begin{prop} \label{prop:MI-iso}
The assignment $I\mapsto \Phi_F(I)$ yields an isomorphism 
$$\Phi_F\colon \mathbf{I}(A)\longrightarrow \mathbf{I}(B)$$ 
between the ideal monoids. Moreover, we have an isomorphism $\Phi_F(I)\simeq F^e(I)$ of $B$-$B$-bimodules.
\end{prop}

\begin{proof}
We identify two-sided ideals of an algebra with its sub bimodules. Consider the following sequence of morphisms between $B$-$B$-bimodules.
$$F^e(I)\hookrightarrow F^e(A)=P\otimes_A A\otimes_AQ\simeq P\otimes_AQ \stackrel{\phi}\longrightarrow B$$
Here, the leftmost morphism is induced by the inclusion $I\hookrightarrow A$. The image of this composite morphism equals $\Phi_F(I)$. This proves the final statement. 

Since $F^e(A)\simeq B$, the Morita equivalence $F^e$ induces a bijection between two-sided ideals of $A$ and those of $B$. This bijection is essentially the same as   $\Phi_F$. Using the associativity condition (\ref{equ:asso}), one verify that $\Phi_F(II')=\Phi_F(I)\Phi_F(I')$ for any two-sided ideals $I$ and $I'$ of $A$.  
\end{proof}

The following result will be useful to determine the isomorphism $\Phi_F$.

\begin{prop}\label{prop:compute-Phi}
Assume that $I$ is a two-sided ideal of $A$ and that $J$ is a two-sided ideal of $B$. Then $\Phi_F(I)=J$ if and only if ${\rm add}\; F(A/I)={\rm add}\; (B/J)$ in $B\mbox{-}{\rm Mod}$.
\end{prop}

\begin{proof}
Set $\Phi_F(I)=J'$. By Proposition~\ref{prop:MI-iso}, we identify $F^e(A)$ with $B$, and $F^e(I)$ with  $J'$. Applying $F^e$ to the following canonical exact sequence 
$$0\longrightarrow I\longrightarrow A\longrightarrow A/I\longrightarrow 0,$$
we infer that $F^e(A/I)\simeq B/{J'}$. By Lemma~\ref{lem:Fe}, we have 
$${\rm add}\; F(A/I)={\rm add}\; (B/{J'})$$
in $B\mbox{-Mod}$. Then the lemma below implies the required statement.
\end{proof}

The following fact is well known.

\begin{lem}
Assume that both $J$ and $J'$ are two-sided ideals of $B$. Then $J=J'$ if and only if ${\rm add}\; (B/J)={\rm add}\; (B/J')$ in $B\mbox{-}{\rm Mod}$.
\end{lem}

\begin{proof}
It suffices to prove the ``if" part. The annihilator ideal of  the left $B$-module $B/J$ equals $J$. Moreover,  for any $E\in {\rm add}\; (B/J)$, $J$ is contained in the annihilator ideal of $E$. Then the required equality follows immediately.
\end{proof}

The following result is implicit in \cite[Proposition~3.10(e)]{Ke11}. 

\begin{lem}\label{lem:Morita-bimod}
Let $F\colon A\mbox{-}{\rm Mod}\rightarrow B\mbox{-}{\rm Mod}$ be the given Morita equivalence. Then for each $i\geq 0$, the following diagram
\[\xymatrix{
A\mbox{-}{\rm Mod} \ar[d]_-{{\rm Ext}_{A^e}^i(A, A^e)\otimes_A-}\ar[rr]^F && B\mbox{-}{\rm Mod} \ar[d]^-{{\rm Ext}_{B^e}^i(B, B^e)\otimes_B-}\\
A\mbox{-}{\rm Mod}\ar[rr]^F && B\mbox{-}{\rm Mod}
}\]
commutes up to a natural isomorphism.
\end{lem}

\begin{proof}
Since $F\simeq P\otimes_A-$, it suffices to show an isomorphism of $B$-$A$-bimodules
$$P\otimes_A {\rm Ext}_{A^e}^i(A, A^e)\simeq {\rm Ext}_{B^e}^i(B, B^e)\otimes_BP.$$
For this end, we observe that the Morita equivalence $F^e$ sends $A$ to $B$, and $A^e$ to $P\otimes Q$. Therefore, we have the following isomorphism of $A$-$A$-bimodules.
\begin{align}\label{iso:AB-bimod}
{\rm Ext}_{A^e}^i(A, A^e)\simeq {\rm Ext}_{B^e}^i(B, P\otimes Q)
\end{align}
Since $P\otimes Q$ is a finitely generated projective $B^e$-module, we have
\begin{align}\label{iso:PQ}
{\rm Ext}_{B^e}^i(B, P\otimes Q)\simeq {\rm Ext}_{B^e}^i(B, B^e)\otimes_{B^e}(P\otimes Q)=Q\otimes_B{\rm Ext}_{B^e}^i(B, B^e)\otimes_BP.
\end{align}
Consequently, we have the following isomorphisms.
\begin{align*}
P\otimes_A {\rm Ext}_{A^e}^i(A, A^e) &\simeq P\otimes_A {\rm Ext}_{B^e}^i(B, P\otimes Q)\\
                     & \simeq P\otimes_A (Q\otimes_B{\rm Ext}_{B^e}^i(B, B^e)\otimes_BP)\\
                     & \simeq {\rm Ext}_{B^e}^i(B, B^e)\otimes_BP
\end{align*}
Here, the first isomorphism uses (\ref{iso:AB-bimod}), the second one uses (\ref{iso:PQ}) and the last one uses the isomorphism $\phi\colon P\otimes_A Q\rightarrow  B$. This completes the proof.
\end{proof}

The following result shows that any Morita equivalence between $A$ and $B$ extends to a Morita equivalence between their $n$-preprojective algebras $\Pi_n(A)$ and $\Pi_n(B)$ for $n\geq 2$. We emphasize that the result is essentially due to \cite[Proposition~4.2]{Ke11}.

\begin{prop}\label{prop:Pi}
  Let $F\colon A\mbox{-{\rm Mod}}\rightarrow B\mbox{-{\rm Mod}}$ be a Morita equivalence and $n\geq 2$. Then there is a Morita equivalence $\widetilde{F}$ making the following diagram commute up to a natural isomorphism.
  \[\xymatrix{
  A\mbox{-{\rm Mod}}\ar[d]\ar[rr]^F && B\mbox{-{\rm Mod}} \ar[d]\\
  \Pi_n(A)\mbox{-{\rm Mod}}\ar[rr]^{\widetilde{F}} && \Pi_n(B)\mbox{-{\rm Mod}}
  } \]
\end{prop}

Here, we identify any $A$-module $X$ with the corresponding $\Pi_n(A)$-module $X$ on which ${\rm Ext}_{A^e}^{n-1}(A, A^e)$ acts trivially. This yields the unnamed vertical arrow on the left side. Similar remarks work for the right side. 

\begin{proof}
    By combining Lemmas~\ref{lem:Morita-bimod} and \ref{lem:rep}, the given Morita equivalence $F$ induces an equivalence
    $$\widetilde{F} \colon {\rm rep}({\rm Ext}_{A^e}^{n-1}(A, A^e)\otimes_A-) \longrightarrow {\rm rep}({\rm Ext}_{B^e}^{n-1}(B, B^e)\otimes_B-).$$ 
    By  the isomorphism (\ref{iso:TAV}), we  identify $\Pi_n(A)\mbox{-Mod}$ with ${\rm rep}({\rm Ext}_{A^e}^{n-1}(A, A^e)\otimes_A-)$,  and $\Pi_n(B)\mbox{-Mod}$ with ${\rm rep}({\rm Ext}_{B^e}^{n-1}(B, B^e)\otimes_B-)$. 
    This completes the proof. 
\end{proof}

\section{Weyl groups and monoids}

In this section, we study Weyl groups and monoids associated to quivers and Cartan matrices. In the folding process, we prove that the Weyl monoid of a Cartan matrix is isomorphic to the invariant monoid of a quiver; see Proposition~\ref{prop:monoid}.  We refer to \cite[Section~3]{Kac} for Weyl groups and \cite{Ts} for Weyl monoids. 

Let $Q$ be a finite acyclic quiver. Denote by $W(Q)$ its \emph{Weyl group}. It is generated by simple reflections $\{s_i\; |\; i\in Q_0\}$, which are subject to the following relations: $s_i^2=1$; $(s_is_j)^2=1$ if there is no arrow between $i$ and $j$; $(s_is_j)^3=1$ if there is precisely one arrow between $i$ and $j$. 

Denote by $WM(Q)$ the \emph{Weyl monoid}, which is a monoid generated by $\{h_i\; |\; i\in Q_0\}$ subject to the following relations $h_i^2=h_i$;  $h_ih_j=h_jh_i$ if there is no arrow between $i$ and $j$; $h_ih_jh_i=h_jh_ih_j$ if there is precisely one arrow between $i$ and $j$. 

By \cite[Theorem~1]{Ts}, there is a bijection 
$$\rho_Q\colon W(Q)\longrightarrow WM(Q)$$
given as follows:  for any reduced expression $\omega=s_{i_1}s_{i_2}\cdots s_{i_n}$ in $W(Q)$ with $i_1, i_2, \cdots, i_n\in Q_0$, we have $\rho_Q(\omega)=h_{i_1}h_{i_2}\cdots h_{i_n}$.

The above consideration works well for Cartan matrices. Let $C=(c_{ij})\in M_n(\mathbb{Z})$ be a \emph{symmetrizable generalized Cartan matrix}. Therefore, the following conditions are fulfilled. 
\begin{enumerate}
    \item[(C1)] $c_{ii}=2$ for each $i$;
    \item[(C2)] $c_{ij}\leq 0$ for all $i\neq j$, and $c_{ij}<0$ if and only if $c_{ji}<0$;
    \item[(C3)] There exists a diagonal matrix $D={\rm diag}(c_1, c_2, \cdots, c_n)$ with each $c_i$ a positive integer such that $DC$ is symmetric. 
\end{enumerate}
Such a matrix $D$ is called a \emph{symmetrizer} of $C$.

Denote by $W(C)$ the associated Weyl group, which is generated by simple reflections $\{r_1, r_2, \cdots, r_n\}$ subject to the following relations: $r_i^2=1$; $(r_ir_j)^2=1$ if $c_{ij}=0$; $(r_ir_j)^3=1$ if $c_{ij}c_{ji}=1$; $(r_ir_j)^4=1$ if $c_{ij}c_{ji}=2$;  $(r_ir_j)^6=1$ if $c_{ij}c_{ji}=3$. 

Similarly, the Weyl monoid $WM(C)$ is a monoid generated by $\{f_1, f_2, \cdots, f_n\}$ subject to the relations: $f_i^2=f_i$; $f_if_j=f_jf_i$ if $c_{ij}=0$; $f_if_jf_i=f_jf_if_j$ if $c_{ij}c_{ji}=1$; $(f_if_j)^2=(f_jf_i)^2$ if $c_{ij}c_{ji}=2$;  $(f_if_j)^3=(f_jf_i)^3$ if $c_{ij}c_{ji}=3$. 

By \cite[Theorem~1]{Ts},  there is a bijection 
$$\rho_C\colon W(C)\longrightarrow WM(C)$$
given as follows:  any reduced expression $\omega=r_{i_1}r_{i_2}\cdots r_{i_m}$ in $W(C)$, we have $\rho_C(\omega)=f_{i_1}f_{i_2}\cdots f_{i_m}$.

\begin{rem}
    For any Coxeter group $W$ with the Coxeter matrix $M$, one defines the corresponding \emph{Coxeter monoid} $WM$ in \cite{Ts}. There is a similar bijection $\rho\colon W\rightarrow WM$; see \cite[Theorem~1]{Ts}. When $M$ arises from a Cartan matrix $C$, the bijection $\rho$ coincides with $\rho_C$. We mention that the monoid algebra of $WM$ is isomorphic to the $0$-Hecke algebra \cite{Nor}; compare \cite[Theorem~4.4]{CLu}. 
\end{rem}

Each finite acyclic quiver $Q$ gives rise to a symmetric Cartan matrix $C$ in the following way: the rows and columns of $C$ are indexed by $Q_0$; for $i\neq j$, we have 
$$c_{ij}=-|\{\mbox{arrows between } i \mbox{ and }j\}|.$$ 
The two Weyl groups $W(Q)$ and $W(C)$ coincide.  Similarly,   the two Weyl monoids $WM(Q)$ and $WM(C)$ coincide. Moreover, $\rho_Q=\rho_C$. 

The following example, known as the folding process,  is our main concern. 

\begin{exm}\label{exm:GCM}
{\rm Let $G$ be a finite group acting on a finite acyclic quiver $Q$. Denote by $Q_0/G$ the set of $G$-orbits, whose elements will be denoted by bold letters. Since $Q$ is acyclic, there is no arrow between any two vertices in the same orbit.

We associate a Cartan matrix $C$ to this action as follows. The rows and columns of $C$ are indexed by $Q_0/G$. The entries are given by
$$c_{\mathbf{i}, \mathbf{j}}=\frac{-N_{\mathbf{i}, \mathbf{j}}}{|\mathbf{j}|}$$
where $N_{\mathbf{i}, \mathbf{j}}$ counts all arrows in $Q$ between the orbits $\mathbf{i}$ and $\mathbf{j}$. The symmetrizer $D={\rm diag}(c_{\mathbf{i}})_{\mathbf{i}\in Q_0/G}$ of $C$ is given such that 
$$c_{\mathbf{i}} =\frac{|G|}{|\mathbf{i}|}.$$ 
The group $G$ acts on $W(Q)$ by group automorphisms. By \cite[\S 11]{St} and \cite[Proposition~3.4]{Hee}, there is a well-known isomorphism of groups
$$\psi\colon W(C)\longrightarrow W(Q)^G,$$
which sends $r_\mathbf{i}$ to $\prod_{i\in \mathbf{i}}s_i$. Here, we denote by $W(Q)^G$ the  fixed subgroup. We refer to  \cite[Lemma~3(3)]{Hu04} and \cite[Theorem~1]{GI} for more recent treatments.}
\end{exm}

The following useful fact can be found in \cite[Lemma~2]{GI}; compare \cite[Proposition~3.4]{Hee}.

\begin{lem}\label{lem:reduced}
Keep the notation above. Let $\omega=r_{\mathbf{i}_1}r_{\mathbf{i}_2}\cdots r_{\mathbf{i}_m}$ be a reduced expression in $W(C)$. Then the expression $\psi(\omega)=(\prod_{i_1\in \mathbf{i}_1}s_{i_1})(\prod_{i_2\in \mathbf{i}_2}s_{i_2})\cdots (\prod_{i_m\in \mathbf{i}_m}s_{i_m})$ is also reduced in $W(Q)$. \hfill $\square$
\end{lem}

The $G$-action on $Q$ induces a $G$-action on $WM(Q)$ by monoid automorphisms. Denote by $MW(Q)^G$ the submonoid formed by $G$-invariant elements. We observe that the bijection $\rho_Q$ is compatible with the $G$-actions. In other words, we have
$$\rho_Q(g(\omega))=g(\rho_Q(\omega))$$
for any $g\in G$ and $\omega\in W(Q)$. Here, we use implicitly the fact that the $G$-action preserves reduced expressions in $W(Q)$. Consequently, we have the restricted bijection
$$\rho_Q^G\colon W(Q)^G\longrightarrow WM(Q)^G.$$

We obtain an analogue of the isomorphism $\psi$ for Weyl monoids in the following result, which  seems to be expected by experts. 

\begin{prop}\label{prop:monoid}
    Keep the assumptions in Example~\ref{exm:GCM}. Then there is a unique isomorphism $\psi'\colon WM(C)\rightarrow WM(Q)^G$ of monoids making the following diagram commute.
    \[\xymatrix{
W(C) \ar[d]_-{\rho_C} \ar[rr]^-{\psi} && W(Q)^G \ar[d]^-{\rho_Q^G}\\
WM(C) \ar[rr]^-{\psi'} && WM(Q)^G
    }\]
\end{prop}

\begin{proof}
The commutativity implies that $\psi'=\rho_Q^G\circ \psi\circ \rho_C^{-1}$, which shows the uniqueness of such a map. It remains to show that this map $\psi'$  is indeed a morphism of monoids. 

We claim that the elements in $\{\psi'(f_\mathbf{i})=\prod_{i\in \mathbf{i}} h_i\;|\; \mathbf{i}\in Q_0/G\}$ satisfy the defining relations of $WM(C)$. Since there are no arrows between vertices in the $G$-orbit $\mathbf{i}$, we infer from $h_i^2=h_i$ that $\psi'(f_\mathbf{i})^2=\psi'(f_\mathbf{i})$. Similarly, if $c_{\mathbf{i}, \mathbf{j}}=0$, there are no arrows between $\mathbf{i}$ and $\mathbf{j}$. Then we deduce that $\psi'(f_\mathbf{i}) \psi'(f_\mathbf{j})=\psi'(f_\mathbf{j}) \psi'(f_\mathbf{i})$.

Assume that $c_{\mathbf{i}, \mathbf{j}}c_{\mathbf{j}, \mathbf{i}}=1$. Since $r_\mathbf{i}r_\mathbf{j}r_\mathbf{i}$ is a reduced expression in $W(C)$, Lemma~\ref{lem:reduced} implies that $(\prod_{i\in \mathbf{i}} s_i)(\prod_{j\in \mathbf{j}} s_j)(\prod_{i\in \mathbf{i}} s_i)$ is a reduced expression in $W(Q)$. Then by the construction of $\rho_Q$, we have the second equality in the following identity.
\begin{align*}
 \psi'(f_\mathbf{i})\psi'(f_\mathbf{j})\psi'(f_\mathbf{i}) & =(\prod_{i\in \mathbf{i}} h_i) (\prod_{j\in \mathbf{j}} h_j) (\prod_{i\in \mathbf{i}} h_i) \\
 &=\rho_Q((\prod_{i\in \mathbf{i}} s_i)(\prod_{j\in \mathbf{j}} s_j)(\prod_{i\in \mathbf{i}} s_i))\\
 &=\rho_Q\psi(r_\mathbf{i}r_\mathbf{j}r_\mathbf{i}).
\end{align*}
Similarly, we have 
$$\psi'(f_\mathbf{j})\psi'(f_\mathbf{i})\psi'(f_\mathbf{j})=\rho_Q\psi(r_\mathbf{j}r_\mathbf{i}r_\mathbf{j}).$$
Since $r_\mathbf{i}r_\mathbf{j}r_\mathbf{i}=r_\mathbf{j}r_\mathbf{i}r_\mathbf{j}$ holds in $W(C)$, we infer the desired identity
$$\psi'(f_\mathbf{i})\psi'(f_\mathbf{j})\psi'(f_\mathbf{i})=\psi'(f_\mathbf{j})\psi'(f_\mathbf{i})\psi'(f_\mathbf{j})$$
in $WM(Q)$. A similar proof works for the cases $c_{\mathbf{i}, \mathbf{j}}c_{\mathbf{j}, \mathbf{i}}=2$ and $3$. This completes the proof of the claim.

By the claim, we have a well-defined morphism 
$$\psi'' \colon WM(C)\longrightarrow WM(Q)^G$$
between monoids such that $\psi''(f_\mathbf{i})=\psi'(f_\mathbf{i})$ for each $\mathbf{i}\in Q_0/G$. By Lemma~\ref{lem:reduced}, it is direct to check that $\psi''\circ \rho_C=\rho_Q^G\circ \psi$. This will force that $\psi''=\psi'$, which completes the proof. 
\end{proof}

Let $Q$ be a finite acylic quiver. For each $i\in Q_0$, write $I_i=\Pi(Q)(1-e_i)\Pi(Q)$; it is a two-sided ideal of $\Pi(Q)$. Denote by $\langle I_i\; |\; i\in Q_0 \rangle$ the sub monoid of $\mathbf{I}(\Pi(Q))$ generated by these ideals $I_i$.

The following isomorphism is  essentially  due to \cite[Theorem~III.1.9]{BIRS} and \cite[Theorem~2.14]{Mizuno}.

\begin{thm}\label{thm:Iyama}
Let $Q$ be a finite acyclic quiver. There is an isomorphism 
$$\Theta'_Q\colon WM(Q)\longrightarrow \langle I_i\; |\; i\in Q_0 \rangle$$
between monoids, which sends $h_i$ to $I_i$ for each $i\in Q_0$.
Consequently, we have a bijection
$$\Theta_Q=\Theta'_Q\circ \rho_Q\colon W(Q)\longrightarrow \langle I_i\; |\; i\in Q_0 \rangle.$$
\end{thm}

\begin{proof}
The well-definedness of the morphism $\Theta_Q'$ is due to \cite[Propostion~III.1.8]{BIRS}. It is clearly surjective. For the injectivity, we refer to the proofs of  \cite[Theorem~III.1.9]{BIRS} and \cite[Theorem~2.14]{Mizuno}.
\end{proof}

Let $G$ be a finite group which acts on $Q$. For each $\mathbf{i}\in Q_0/G$, we write 
$$I_{\mathbf{i}}=\prod_{i\in \mathbf{i}} I_i.$$ 
This is well defined, since $I_i$ and $I_j$ commute for $i,j\in \mathbf{i}$; see \cite[Propostion~III.1.8]{BIRS}. We observe that $I_{\mathbf{i}}$ is a $G$-invariant ideal of $\Pi(Q)$. Then $I_{\mathbf{i}}\#G$ is a two-sided ideal of $\Pi(Q)\#G$; see Proposition~\ref{prop:inv-ideal}. Denote by $\langle I_{\mathbf{i}}\#G\; |\; \mathbf{i}\in Q_0/G\rangle$ the sub monoid of $\mathbf{I}(\Pi(Q)\#G)$ generated by $I_{\mathbf{i}}\#G$.

\begin{prop}\label{prop:Iyama-G}
    Let $G$ be a finite group which acts on a finite acyclic quiver $Q$. Then there is an isomorphism of monoids $WM(Q)^G\rightarrow \langle I_{\mathbf{i}}\#G\; |\; \mathbf{i}\in Q_0/G\rangle$.
\end{prop}

\begin{proof}
The bijection $\rho_Q$ restricts to the bijection $\rho_Q^G\colon W(Q)^G\rightarrow WM(Q)^G$. Since $W(Q)^G$ is generated by $\{\prod_{i\in \mathbf{i}}s_i\;|\; \mathbf{i}\in Q_0/G\}$, it follows that the monoid $WM(Q)^G$ is generated by $\{\prod_{i\in \mathbf{i}}h_i\;|\; \mathbf{i}\in Q_0/G\}$.  The isomorphism in Theorem~\ref{thm:Iyama} implies the following observation: 
$\langle I_i\; |\; i\in Q_0\rangle^G$ is generated by $\{I_\mathbf{i}\; |\; \mathbf{i}\in Q_0/G\}$. 

By the isomorphism in Proposition~\ref{prop:inv-ideal}(3), we have an injective morphism of monoids
$$\langle I_i\; |\; i\in Q_0\rangle^G  \hookrightarrow \mathbf{I}(\Pi(Q)\#G), \; I \mapsto I\#G.$$
The observation above implies that the image of this morphism is precisely $\langle I_{\mathbf{i}}\#G\; |\; \mathbf{i}\in Q_0/G\rangle$. Combining this with  the isomorphism $\Theta'_Q$ in Theorem~\ref{thm:Iyama}, we complete the proof.\end{proof}

\begin{rem}\label{rem:monoid}
Recall from Example~\ref{exm:GCM} that the Cartan matrix $C$ is associated to the $G$-action. Recall that $\Theta_Q=\Theta'_Q \circ \rho_Q$. By the commutative square in Proposition~\ref{prop:monoid}, we have the following commutative square consisting of bijections. 
\[\xymatrix{
W(C)\ar[d] \ar[rr]^-{\psi} && W(Q)^G\ar[d]^-{\Theta_Q^G}\\
\langle I_\mathbf{i}\# G\; |\; \mathbf{i}\in Q_0/G\rangle && \ar[ll]_-{-\#G} \langle I_i\; |\; i\in Q_0\rangle^G
}\]
Here, the unnamed vertical arrow is given by $(-\#G)\circ (\Theta'^G_Q)\circ \psi'\circ \rho_C$.
\end{rem}

\section{Algebras associated to Cartan triples}
In this section, we recall the generalized preprojective algebras from \cite{GLS17} and prove Proposition~\ref{prop:Iyama-FG}, which is Proposition~A in Introduction.

Recall from \cite{CW24} that a \emph{Cartan triple} $(C, D, \Omega)$ consists of a Cartan matrix $C=(c_{ij})\in M_n(\mathbb{Z})$, its symmetrizer $D={\rm diag}(c_1, c_2, \cdots, c_n)$ and an acyclic orientation $\Omega$.  Here, we recall that an  \emph{acyclic orientation} $\Omega$ on $C$ is a subset of $\{1, 2, \cdots, n\}\times \{1, 2, \cdots, n\}$ subject to the following conditions.
\begin{enumerate}
\item[(O1)] $\{(i, j), (j, i)\}\cap \Omega  \neq \emptyset$ if and only if $c_{ij}<0$; 
\item[(O2)] for each sequence $(i_1, i_2, \cdots, i_{t}, i_{t+1})$ with $t\geq 1$ such that $(i_s, i_{s+1})\in \Omega$ for all $1\leq s\leq t$, we necessarily have $i_1 \neq i_{t+1}$.
\end{enumerate}

Let $Q=Q(C,\Omega)$ be the finite quiver with the set of vertices $Q_0=\{1,2, \cdots,n\}$ and with the set of arrows
\[Q_1=\{\alpha^{(g)}_{ij}\colon j\rightarrow i\mid(i,j)\in \Omega, 1\leq g\leq {\rm gcd}(c_{ij},c_{ji})\}\sqcup\{\varepsilon_i\colon i\rightarrow i\mid 1\leq i\leq n\}.\]
Here, for any nonzero integers $c, c'$, their greatest common divisor ${\rm gcd}(c, c')$ is defined to be positive.

Following \cite[Subsection~1.4]{GLS17} and \cite{Gei}, we associate a finite dimensional algebras $H(C,D,\Omega)$ to any Cartan triple $(C, D, \Omega)$ as follows
\[H(C,D,\Omega)=\mathbb{K}Q/I,\]
where $\mathbb{K}Q$ is the path algebra of $Q=Q(C, \Omega)$, and $I$ is the two-sided ideal of $\mathbb{K}Q$ generated by  the following set
$$\; \{\varepsilon_k^{c_k}, \; \varepsilon_i^{\frac{c_i}{{\rm gcd}(c_i,c_j)}} \alpha^{(g)}_{ij}-\alpha^{(g)}_{ij}\varepsilon_j^{\frac{c_j}{{\rm gcd}(c_i,c_j)}}\mid k\in Q_0, \;  (i,j)\in \Omega, \; 1\leq g\leq {\rm gcd}(c_{ij},c_{ji})\}.$$
Here, $\varepsilon_k^{c_k}$ is called the \emph{nilpotency relation}, and $\varepsilon_i^{\frac{c_i}{{\rm gcd}(c_i,c_j)}} \alpha^{(g)}_{ij}-\alpha^{(g)}_{ij}\varepsilon_j^{\frac{c_j}{{\rm gcd}(c_i,c_j)}}$ the \emph{commuativity relation}. 

Let $(C, D, \Omega)$ be a Cartan triple. The opposite orientation of $\Omega$ is $\Omega^{\rm op}=\{(j, i)\; |\; (i, j)\in \Omega\}$. Set $\overline{\Omega}=\Omega\cup \Omega^{\rm op}$. For each  $(i, j)\in \overline{\Omega}$, we define
\begin{align*}
    {\rm sgn}(i, j)=\left \{ \begin{array}{cc}
         1, \mbox{ if } (i, j)\in \Omega;  \\
         -1, \mbox{ if } (i, j)\in \Omega^{\rm op}. 
    \end{array} \right. 
\end{align*}

Denote by $\widetilde{Q}=\widetilde{Q}(C, \Omega)$ the quiver obtained from $Q=Q(C, \Omega)$ by adding a new arrow $\alpha_{ji}^{(g)}\colon i\rightarrow j$ for each arrow $\alpha_{ij}^{(g)}\colon j\rightarrow i$.  We mention that $\widetilde{Q}$ is not the double quiver of $Q$, since we do not double the loops. 

\begin{defn}
Let $(C, D, \Omega)$ be a Cartan triple and set $\widetilde{Q}=\widetilde{Q}(C, \Omega)$. The \emph{generalized preprojective algebra} is defined to be 
$$\Pi(C, D, \Omega)=\mathbb{K}\widetilde{Q}/{\widetilde{I}},$$
where the two-sided ideal $\widetilde{I}$ of $\mathbb{K}\widetilde{Q}$ is determined by the following relations.
\begin{enumerate}
    \item[(P1)] For each vertex $i\in Q_0$, we have the nilpotency relation $\varepsilon_i^{c_i}=0$.
    \item[(P2)] For each $(i, j)\in \overline{\Omega}$ and $1\leq g\leq {\rm gcd}(c_{ij}, c_{ji})$, we have the commutativity relation $\varepsilon_i^{\frac{c_i}{{\rm gcd}(c_i,c_j)}} \alpha^{(g)}_{ij}-\alpha^{(g)}_{ij}\varepsilon_j^{\frac{c_j}{{\rm gcd}(c_i,c_j)}}$.
    \item[(P3)] For each vertex $i\in Q_0$, we have the \emph{mesh relation}
    $$\sum_{\{j\in Q_0\; |\; (i, j)\in \overline{\Omega}\}} \sum_{g=1}^{{\rm gcd}(c_{ij}, c_{ji})} \sum_{l=0}^{\frac{c_i}{{\rm gcd}(c_i, c_j)}-1} {\rm sgn}(i,j)\; \varepsilon_i^l \alpha_{ij}^{(g)}\alpha_{ji}^{(g)}\varepsilon_i^{\frac{c_i}{{\rm gcd}(c_i, c_j)}-1-l}=0.$$
\end{enumerate}
\end{defn}

The following result is essentially due to \cite[Theorem~1.6]{GLS17}.

\begin{prop}\label{prop:GLS}
    Let $(C, D, \Omega)$ be a Cartan triple with $H=H(C, D, \Omega)$. Then we have an isomorphism  of algebras
    $$\Pi(C, D, \Omega)\simeq \Pi_2(H).$$ 
\end{prop}

\begin{proof}
Since the algebra $H$ is finite dimensional, by Lemma~\ref{lem:fd} we have an isomorphism  ${\rm Ext}^1_{H^e}(H, H^e)\simeq {\rm Ext}_H^1(DH, H)$ of $H$-$H$-bimodules. Then the required isomorphism follows from \cite[Theorem~1.6]{GLS17}.
\end{proof}

For each vertex $i\in Q_0$, we denote by $e_i$ the corresponding idempotent of $\Pi(C, D, \Omega)$. Denote by $L_i$ the two-sided ideal generated by $1-e_i$. Denote by $\langle L_i\; |\; i\in Q_0 \rangle$ the sub monoid of $\mathbf{I}(\Pi(C, D, \Omega))$ generated by these $L_i$'s. Consider the Weyl group $W(C)$, whose simple reflections are denoted by $r_i$ for each  $ i\in Q_0$.

The following result is essentially due to \cite[Theorem~4.7]{FG}.

\begin{thm}\label{thm:FG}
Let $(C, D, \Omega)$ be a Cartan triple. There is an isomorphism 
$$\Theta'_C\colon WM(C)\longrightarrow \langle L_i\; |\; i\in Q_0 \rangle$$
of monoids, which sends $f_i$ to $L_i$. Consequently, we have a bijection
$$\Theta_C=\Theta'_C\circ \rho_C\colon W(C)\longrightarrow \langle L_i\; |\; i\in Q_0 \rangle.$$
\end{thm}

\begin{proof}
    The well-definedness of the morphism $\Theta_C'$ is due to \cite[Proposition~4.6]{FG}, which is clearly surjective. For the injectivitiy, we refer to the proof of \cite[Theorem~4.7]{FG}.
\end{proof}

In the example below, we see that the ismorphism in Theorem~\ref{thm:FG} extends the one in Theorem~\ref{thm:Iyama}.

\begin{exm}
{\rm Let $Q$ be a finite acyclic quiver. Then it corresponds to a symmetric Cartan matrix $C$. Denote by $I_{Q_0}$ the identity matrix with rows and columns indexed by $Q_0$. The set  $Q_1$ of arrows yields an acyclic orientation $\Omega$ on $Q_0$ in the obvious manner: there is an arrow from $i$ to $j$ in $Q$ if and only if $(j, i)$ belongs to $\Omega$. 

We have algebra isomorphisms $\mathbb{K}Q\simeq H(C, I_{Q_0}, \Omega)$ and $\Pi(Q)
\simeq \Pi(C, I_{Q_0}, \Omega)$. Since $WM(Q)=WM(C)$, therefore in this situation, the isomorphisms in Theorem~\ref{thm:FG} coincides with the one in Theorem~\ref{thm:Iyama}. }
\end{exm}

Let $G$ be a finite group, which acts on a finite acyclic quiver $Q$. Following the folding process in Example~\ref{exm:GCM}, we denote by $(C, D, \Omega)$ the associated Cartan triple. Here, the rows and columns of $C$ and $D$ are indexed by the orbit set $Q_0/G$. The acyclic orientation $\Omega$ is defined to such that $(\mathbf{j}, \mathbf{i})$ belongs to $\Omega$ if and only if there is some arrow from the orbit $\mathbf{i}$ to the orbit $\mathbf{j}$ in $Q$. We mention that each Cartan triple arises in this way; compare \cite[Remark~6.9]{CW24} and \cite[Section~14.1]{Lust}

Write $\Pi=\Pi(C, D, \Omega)$. Each vertex $\mathbf{i}$ in $Q(C, \Omega)_0=Q_0/G$ gives rise to an idempotent $e_\mathbf{i}$ of  $\Pi$. Recall that $L_\mathbf{i}=\Pi (1-e_\mathbf{i})\Pi$.

The following square compares the bijections obtained in \cite{BIRS, Mizuno} and \cite{FG}.

\begin{prop}\label{prop:Iyama-FG}
  Keep the assumptions above. Then there is a unique isomorphism $\Psi$ between monoids making the following diagram commute. 
\[\xymatrix{
W(C)\ar[rr]^\psi \ar[d]_-{\Theta_C} && W(Q)^G \ar[d]^-{\Theta_Q^G}\\
\langle L_\mathbf{i}\; |\; \mathbf{i}\in Q_0/G\rangle  \ar[rr]^-{\Psi} && \langle  I_i\; |\; i\in Q_0\rangle^G
}\]
\end{prop} 

We mention that $\Psi(L_\mathbf{i})=I_\mathbf{i}$ for each $\mathbf{i}\in Q_0/G$.

\begin{proof}
   The uniqueness of $\Psi$ is clear, since the other three maps in the square are all bijections. Recall that $\Theta_C=\Theta'_C\circ \rho_C$ and $\Theta_Q=\Theta'_Q\circ \rho_Q$. In view of the commutative square in Proposition~\ref{prop:monoid}, it suffices to take 
   $$\Psi=\Theta'^G_Q\circ \psi'\circ (\Theta'_C)^{-1}.$$ 
   As  a composition of three isomorphisms between monoids, $\Psi$ is an isomorphism of  monoids. 
\end{proof}

\section{The main result} 

In this section, we establish a Morita equivalence involving preprojective algebras in the folding process.  We assume that ${\rm char}(\mathbb{K})=p>0$ and that $G=\langle \sigma\; |\; \sigma^{p^a}=1\rangle$ is a cyclic $p$-group for some $a\geq 1$.

Let $Q$ be a finite acyclic quiver. For each vertex $i\in Q_0$, we denote by $S_i$ the $1$-dimensional simple $\mathbb{K}Q$-module concentrated in the vertex $i$. It is naturally a simple $\Pi(Q)$-module. Recall that $I_i=\Pi(Q)(1-e_i)\Pi(Q)$. Then we have an exact sequence
\begin{align*}
0\longrightarrow I_i\longrightarrow \Pi(Q)\longrightarrow S_i\longrightarrow 0
\end{align*}
of left $\Pi(Q)$-modules.

Fix a $G$-action on $Q$. We assume that the action satisfies the following condition.
\vskip 3pt
\noindent $(*)$ \;  For each arrow $\alpha$ in $Q$, we have $G_{\alpha}=G_{s(\alpha)}\cap G_{t(\alpha)}$. Here, $G_{\alpha}$, $G_{s(\alpha)}$ and $G_{t(\alpha)}$ denote their stabilizers. 

\vskip 3pt

We denote by $(C, D, \Omega)$ the associated Cartan triple. Write $H=H(C, D, \Omega)$ and $\Pi=\Pi(C, D, \Omega)$. Each vertex $\mathbf{i}$ in $Q(C, \Omega)_0=Q_0/G$ gives rise to an idempotent $e_\mathbf{i}$ of $H$, which is also an idempotent of $\Pi$. Following \cite[Subsection~3.2]{GLS17}, we denote by $E_{\mathbf{i}}=H/{H(1-e_\mathbf{i})H}$ the \emph{generalized simple module} over $H$ concentrated in $\mathbf{i}$. We mention that $E_\mathbf{i}$ is naturally isomorphic to $\mathbb{K}[\varepsilon]/{(\varepsilon^{c_{\mathbf{i}}})}$.  We view $E_\mathbf{i}$ as a module over $\Pi$. Recall that $L_\mathbf{i}=\Pi(1-e_\mathbf{i})\Pi$. Then we have an exact sequence  
\begin{align}\label{ses:L}
0\longrightarrow L_\mathbf{i}\longrightarrow \Pi\longrightarrow E_\mathbf{i}\longrightarrow 0
\end{align}
of left $\Pi$-modules. 

The following Morita equivalence is due to \cite{CW24}, which is based on \cite{CW}. 

\begin{prop}\label{prop:CW}
    Keep the assumptions above. Then there is a Morita equivalence 
    $$U\colon \mathbb{K}Q\#G\mbox{-}{\rm Mod}\longrightarrow H(C, D, \Omega)\mbox{-}{\rm Mod}$$
    such that $U((\mathbb{K}Q\#G)\otimes_{\mathbb{K}Q} S_i)\simeq E_{\mathbf{i}}$ for each $\mathbf{i}\in Q_0/G$ and $i\in \mathbf{i}$.
\end{prop}

\begin{proof}
This is due to \cite[Theorem~7.8]{CW24}. We observe that the folding projection $\mathbf{f}$ therein sends simple roots to simple roots. It follows that the Morita equivalence $U$ sends $(\mathbb{K}Q\#G)\otimes_{\mathbb{K}Q} S_i$ to $E_\mathbf{i}$. 
\end{proof}

Recall the bijection $\Theta_Q\colon W(Q)\rightarrow \langle I_i\;|\; i\in Q_0\rangle $ in Section~5 and the bijection $\Theta_C\colon W(C)\rightarrow \langle L_\mathbf{i}\; |\; \mathbf{i}\in Q_0/G\rangle$ in Section~6. The bijection
$\Theta_Q$ restricts to a bijection 
$$\Theta_Q^G\colon W(Q)^G\longrightarrow \langle I_i\;|\; i\in Q_0\rangle^G $$ between the subsets formed by $G$-invariant elements. Recall from Example~\ref{exm:GCM} the isomorphism $\psi\colon W(C)\rightarrow W(Q)^G$, which sends the simple reflections $r_\mathbf{i}$ to $\prod_{j\in \mathbf{i}}s_j$.

The main result of this work is as follows. 

\begin{thm}\label{thm:main}
Assume that ${\rm char}(\mathbb{K})=p>0$ and that $G$ is a cyclic $p$-group which acts on a finite acyclic quiver $Q$ satisfying $(*)$. Keep the notation above. Then there is a Morita equivalence
$$F\colon \Pi(Q)\#G\mbox{-}{\rm Mod}\longrightarrow \Pi(C, D, \Omega)\mbox{-}{\rm Mod}$$
such that the following diagram commutes.
\[\xymatrix{
W(C) \ar[d]_-{\Theta_C}\ar[rr]^-{\psi} && W(Q)^G \ar[d]^-{\Theta_Q^G}\\
\langle L_\mathbf{i}\; |\; \mathbf{i}\in Q_0/G\rangle  && \ar[ll]_-{\Phi_F(-\#G)} \langle I_i\; |\; i\in Q_0 \rangle^G
}\]
\end{thm}

For the bottom map $\Phi_F(-\#G)$, we observe that  each element $J$ in $\langle I_i\; |\; i\in Q_0 \rangle^G$ is $G$-invariant, and that $J\#G$ is a two-sided ideal of $\Pi(Q)\#G$. The isomorphism $\Phi_F\colon \mathbf{I}(\Pi(Q)\#G)\rightarrow \mathbf{I}(\Pi(C,D, \Omega))$ is given in Proposition~\ref{prop:MI-iso}. 

\begin{proof}
Write $H=H(C, D, \Omega)$.  We divide the proof into four steps.

\vskip 3pt

    \emph{Step 1.}   By Lemma~\ref{lem:quiver}, we identify $\Pi(Q)$ with the $2$-preprojective algebra $\Pi_2(\mathbb{K}Q)$. Therefore, by Theorem~\ref{thm:pi} the skew group algebra $\Pi(Q)\#G$ is isomorphic to $\Pi_2(\mathbb{K}Q\#G)$. Applying Proposition~\ref{prop:Pi} to the Morita equivalence in Proposition~\ref{prop:CW}, we obtain a Morita equivalence
    $$\widetilde{U}\colon \Pi_2(\mathbb{K}Q\#G)\mbox{-Mod}\longrightarrow \Pi_2(H)\mbox{-Mod}.$$
    By Proposition~\ref{prop:GLS}, we identify $\Pi_2(H)$ with $\Pi(C, D, \Omega)$. Combining these isomorphisms and $\widetilde{U}$, we obtain the required Morita equivalence $F$.

\vskip 3pt

       \emph{Step 2.} We claim that 
       $$F((\Pi(Q)\#G)\otimes_{\Pi(Q)} S_i)\simeq E_\mathbf{i}$$
       for each $\mathbf{i}\in Q_0/G$ and $i\in \mathbf{i}$.  

       For the claim, we recall the isomorphism  
       $$\theta \colon U((\mathbb{K}Q\#G)\otimes_{\mathbb{K}Q} S_i)\rightarrow  E_\mathbf{i}$$
       in $H\mbox{-Mod}$. We view $\mathbb{K}Q\#G$ as a quotient algebra of $\Pi(Q)\#G$, and $H$ as a quotient algebra of $\Pi(C, D, \Omega)$. Then $\theta$ might be viewed as an isomorphism in $\Pi(C,  D, \Omega)\mbox{-Mod}$. The natural map 
       $$(\Pi(Q)\#G)\otimes_{\Pi(Q)} S_i\longrightarrow (\mathbb{K}Q\#G)\otimes_{\mathbb{K}Q} S_i$$
is an isomorphism, since both vector spaces are naturally identified with $\mathbb{K}G\otimes S_i$; compare \cite[(7.1)]{CW24}. Since $\widetilde{U}$ extends $U$, we deduce the claim.
\vskip 3pt

\emph{Step 3.}  For each $\mathbf{i}\in Q_0/G$, we recall that $I_\mathbf{i}=\prod_{j\in \mathbf{i}} I_j$, which is a $G$-invariant ideal of $\Pi(Q)$. We claim that $\Phi_F(I_\mathbf{i}\#G)=L_\mathbf{i}$.

Since the ideal $I_{\mathbf{i}}$ is generated by $\{e_l\;|\; l\notin \mathbf{i}\}$, the quotient algebra $\Pi(Q)/{I_{\mathbf{i}}}$ is isomorphic to a product of $\mathbb{K}$ indexed by $\mathbf{i}$. Consequently, as a left $\Pi(Q)$-module, we have 
\begin{align}\label{iso:Pi-mod}
    \Pi(Q)/{I_{\mathbf{i}}}\simeq \bigoplus_{j\in \mathbf{i}} S_j.
\end{align}
Here, we mention that $\Pi(Q)/{I_{\mathbf{i}}}$ is naturally a $G$-compatible bimodule over $\Pi(Q)$. Therefore, we have the induced bimodule  $(\Pi(Q)/{I_{\mathbf{i}}})\#G$ over $\Pi(Q)\#G$.

By combining Remark~\ref{rem:induced} and (\ref{iso:Pi-mod}), we have the second isomorphism in the following identity consisting of isomorphisms in $\Pi(C, D, \Omega)\mbox{-Mod}$.
\begin{align*}
    F((\Pi(Q)\#G)/{(I_\mathbf{i}\#G)}) &\simeq F((\Pi(Q)/{I_{\mathbf{i}}})\#G)\\
    & \simeq \bigoplus_{j\in \mathbf{i}} \; F((\Pi(Q)\#G)\otimes_{\Pi(Q)} S_j)\\
    & \simeq \bigoplus_{j\in \mathbf{i}} \; E_\mathbf{i} \simeq \bigoplus_{j\in \mathbf{i}} \; \Pi(C, D, \Omega)/{L_\mathbf{i}}
\end{align*}
Here, the third isomorphism uses the claim in Step 2, and the last one follows from (\ref{ses:L}). Apply Proposition~\ref{prop:compute-Phi}, we infer that $\Phi_F(I_\mathbf{i}\#G)=L_\mathbf{i}$, as required. 
\vskip 3pt

\emph{Step 4.} Recall from the proof of Proposition~\ref{prop:Iyama-G} that 
$$\langle I_i\; |\; i\in Q_0\rangle^G=\langle I_\mathbf{i}\; |\; \mathbf{i}\in Q_0/G\rangle.$$ 
Then by Step 3, the following composition of morphisms between monoids is well-defined.
$$\langle I_i\; |\; i\in Q_0\rangle^G  \xrightarrow{-\#G} \langle I_\mathbf{i}\#G\; |\; \mathbf{i}\in Q_0/G\rangle \xrightarrow{\Phi_F} \langle L_\mathbf{i}\; |\; \mathbf{i}\in Q_0/G \rangle$$
By  the isomorphisms in Propositions~\ref{prop:inv-ideal}(3) and ~\ref{prop:MI-iso}, this composite morphism is an isomorphism. Since it sends $I_\mathbf{i}$ to $L_\mathbf{i}$, we infer that it coincides with the inverse of $\Psi$ in Proposition~\ref{prop:Iyama-FG}. Now, the required commutativity follows from Proposition~\ref{prop:Iyama-FG} immediately.
\end{proof}

\begin{rem}
\begin{enumerate}
	\item It might be of interest to compare the two commutative squares in Remark~\ref{rem:monoid} and Theorem~\ref{thm:main}.
	\item When $Q$ is of type $A$ and $G$ is of order $2$, such a  Morita equivalence $F$ is also established in \cite[Lemma 5.4]{KKKMM24} by different method independently.
	\end{enumerate}	
\end{rem}

We illustrate the main result with an explicit example.

\begin{exm}
{\rm 	Let $\mathbb{K}$ be a field of characteristic two, and let $Q$ be the following quiver of type $A_3$.
	\[\xymatrix{ 2 \ar[rr]^-{\alpha}&& 1   && 2' \ar[ll]_-{\alpha'}}\]
 The preprojective algebra $\Pi(Q)$ is given by the following quiver
 	\[\xymatrix{ 2  \ar@/^0.5pc/[rr]^-{\alpha} &&  1 \ar@/^0.5pc/[ll]^-{\alpha^*} \ar@/_0.5pc/[rr]_-{\alpha'^*}   &&  2' \ar@/_0.5pc/[ll]_-{\alpha'}}\]
 subject to the relations $\alpha^*\alpha=0=\alpha'^*\alpha'$ and $ \alpha\alpha^*+\alpha'\alpha'^*=0$. Let $G=\{1_G, \sigma\}$ be a cyclic group of order two, and let $\sigma$ act on $Q$ by interchanging $\alpha$ and $\alpha'$. This action extends a $G$-action on $\Pi(Q)$ by $\sigma(\alpha^*)=\alpha'^*$ and $\sigma(\alpha'^*)=\alpha^*$.

	The associated Cartan triple $(C, D, \Omega)$ is of type $B_2$ and  given as follows:
	$$C=\begin{pmatrix}2 & -1 \\
	-2 & 2\end{pmatrix}, \; D={\rm diag}(2, 1), \mbox{ and } \Omega=\{(\mathbf{1},\mathbf{2})\}.$$
	The generalized preprojective algebra $\Pi(C, D, \Omega)$ is given by the following quiver
	
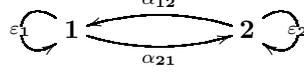
\[
	\xymatrix{
		\boldsymbol{1} \ar@/_0.5pc/[rr]_-{\alpha_{\boldsymbol{21}}}\ar@(ul,dl)[]|{\varepsilon_{\boldsymbol{1}}}  && \ar@/_0.5pc/[ll]_-{\alpha_{\boldsymbol{12}}} \boldsymbol{2}  \ar@(ur,dr)[]|{\varepsilon_{\boldsymbol{2}}}
	}\]
	subject to relations $\varepsilon_{\boldsymbol{1}}^2=0=\varepsilon_{\boldsymbol{2}}$, $\varepsilon_{\boldsymbol{1}}\alpha_{\boldsymbol{12}}\alpha_{\boldsymbol{21}}+\alpha_{\boldsymbol{12}}\alpha_{\boldsymbol{21}}\varepsilon_{\boldsymbol{1}}=0$, and $\alpha_{\boldsymbol{21}}\alpha_{\boldsymbol{12}}=0$. In practice, we omit the loop $\varepsilon_{\mathbf{2}}$. Theorem~\ref{thm:main} yields a Morita equivalence $F$ between $\Pi(Q)\#G$ and $\Pi(C, D, \Omega)$.

 Consider the isomorphism $\psi\colon W(C)\rightarrow W(Q)^G$, which sends $r_\mathbf{1}$ to $s_1$, and $r_\mathbf{2}$ to $s_2s_{2'}=s_{2'}s_2$. We have 
 $$\Phi_F(I_1\#G)=L_\mathbf{1} \mbox{ and } \Phi_F((I_2I_{2'})\#G)= L_\mathbf{2}.$$ The bijection $\Theta_C$ sends the longest element $r_\mathbf{1}r_\mathbf{2}r_\mathbf{1}r_\mathbf{2}$ to the zeo ideal, that is, $$L_\mathbf{1}L_\mathbf{2}L_\mathbf{1}L_\mathbf{2}=0$$
 in $\Pi(C, D, \Omega)$; see \cite[Proposition~4.2]{FG}.  Accordingly, the bijection $\Theta_Q$ sends the longest element $\psi(r_\mathbf{1}r_\mathbf{2}r_\mathbf{1}r_\mathbf{2})=s_1s_2s_{2'}s_1s_2s_{2'}$ to the zero ideal, that is, 
 $$I_1I_2I_{2'}I_1I_2I_{2'}=0$$
 holds in $\Pi(Q)$; compare \cite[Theorem~2.30]{Mizuno}.
}\end{exm}

\vskip 5pt
\noindent {\bf Acknowledgements.}\quad  We thank Professor Ming Fang for helpful suggestions and the reference \cite{Nor}, and Professor Yuya Mizuno for the reference \cite{KKKMM24}. This work is supported by National Natural Science Foundation of China (No.s 12325101, 12131015 and 12161141001).

\bibliography{}

\vskip 10pt

 {\footnotesize \noindent Xiao-Wu Chen\\
 Key Laboratory of Wu Wen-Tsun Mathematics, Chinese Academy of Sciences,\\
 School of Mathematical Sciences, University of Science and Technology of China, Hefei 230026, Anhui, PR China}
 \vskip 5pt

  {\footnotesize \noindent Ren Wang\\
School of Mathematics, Hefei University of Technology, Hefei 230000, Anhui, PR China}

\end{document}